\documentclass{amsart}

\usepackage{amssymb}
\usepackage{amsmath}
\usepackage{amsthm}

\usepackage{graphicx}
 \newtheorem{theorem}{Theorem}[section]
 \newtheorem{cory}[theorem]{Corollary}

 \theoremstyle{definition}
 \newtheorem{df}[theorem]{Definition}
 \theoremstyle{remark}

 \numberwithin{equation}{section}

\newcommand{\al}{\alpha}
\newcommand{\vti}{\widetilde{v}}
\newcommand{\we}{\omega}
\newcommand{\weti}{\widetilde{\we}}
\newcommand{\vot}{V}
\newcommand{\kom}{T}
\newcommand{\kkom}{K}

\newcommand{\lbd}{\nu_d}
\newcommand{\lbarea}{\nu_1}

\newcommand{\vtgp}{V_g^\varphi}
\newcommand{\fk}{\mathcal{F}}

\newcommand{\ph}{\varphi}
\newcommand{\cph}{C_\varphi}
\newcommand{\cmpg}{C_{\varphi, g}}

\newcommand{\hol}{\mathcal{H}ol}

\newcommand{\Dbb}{\mathbb D}
\newcommand{\YY}{\mathcal L}

\newcommand{\cd}{{\mathbb{C}}^d}
\newcommand{\Rbb}{\mathbb R}

\newcommand{\Cbb}{\mathbb C}
\newcommand{\Nbb}{\mathbb N}

\newcommand{\rad}{\mathcal R}

\newcommand{\grw}{\mathcal{A}^{\we}}
\newcommand{\grass}{\mathcal{A}^{\weti}}

\begin{document}

%
%
%
%
%
%
%
%
%

\title[Volterra type operators on growth Fock spaces]
{Volterra type operators\\ on growth Fock spaces}

\author[E.~Abakumov]{Evgeny Abakumov}

\address{%
Universit\'e Paris-Est\\
LAMA (UMR 8050)\\
77454 Marne-la-Vall\'ee\\
France}

\email{evgueni.abakoumov@u-pem.fr}

\thanks{The second author was partially supported by RFBR (grant No. 14-01-00198-a)}
\author[E.~Doubtsov]{Evgueni Doubtsov}
\address{%
St.~Petersburg Department
of V.A.~Steklov Mathematical Institute\\
Fontanka 27\\
St.~Petersburg 191023\\
Russia}
\email{dubtsov@pdmi.ras.ru}



\begin{abstract}
Let $\omega$ be an unbounded radial weight on $\mathbb{C}^d$, $d\ge 1$.
Using results related to approximation of $\we$ by entire maps,
we investigate Volterra type and weighted composition operators defined on the growth space $\mathcal{A}^\omega(\mathbb{C}^d)$.
Special attention is given to the operators defined on the growth Fock spaces.
\end{abstract}

\maketitle

\section{Introduction}\label{s_int}

Let $\hol(\cd)$ denote the space of entire functions on $\cd$, $d\ge 1$.
Given a symbol $g\in \hol(\cd)$, the extended Ces\`aro operator $\vot_g$
is defined as
\[
\vot_g f(z) = \int_0^1 f(tz) \rad g(tz)\, \frac{dt}{t},
\]
where $f\in\hol(\cd)$, $z\in\cd$ and
$\rad g (z) = \sum_{j=1}^d z_j \frac{\partial g}{\partial z_j}(z)$
is the radial derivative of $g$. Relations between the function theoretic properties of the symbol $g$
and the operator theoretic properties of $\vot_g$ have been intensively investigated after
the works of Pommerenke \cite{Po77}, and Aleman and Siskakis \cite{AS97} for $d=1$.
Given a holomorphic map $\ph: \cd \to \cd$, we also consider the composition operator $\cph: f \mapsto f\circ\ph$.
The superpositions
$\vtgp = \vot_g \circ \cph$ and $\cmpg = \cph \circ \vot_g$
are called Volterra type operators.

\subsection{Fock spaces}\label{ss_fock}
The present paper is motivated by recent studies of $\vot_g$, $\vtgp$, $\cmpg$
and related operators defined on or mapping to the Fock type spaces
$\fk_\al^p(\cd)$, $\al>0$, $0<p\le \infty$; see, for example,
\cite{Coo12, Hu13, Men13IEOT, Men14JGA, Men16pota, U07, U16}.
The weighted Fock space $\fk_\al^p(\cd)$, $d\ge 1$, $0<p<\infty$, $\al>0$, consists of those
$f\in\hol(\cd)$ for which
\[
\|f\|_{\fk_\al^p}^p = \left(\frac{\al p}{2\pi} \right)^d \int_{\cd} |f(z)|
^p e^{-\frac{\al p}{2} |z|^2} \, d\lbd (z) < \infty,
\]
where $\lbd$ denotes Lebesgue measure on $\cd$.
For $p=\infty$, the corresponding growth Fock space $\fk_\al^\infty(\cd)$
consists of those
$f\in\hol(\cd)$ for which
\[
\|f\|_{\fk_\al^\infty} = \sup_{z\in\cd} |f(z)| e^{-\frac{\al}{2}|z|^2} < \infty.
\]
The operators $\vtgp: \fk_\al^p(\Cbb)\to \fk_\al^q(\Cbb)$ with $p=\infty$ or $q=\infty$
were investigated by Mengestie \cite{Men13IEOT}.
In fact, the case $q=\infty$ fits into certain general schemes. For example,
given $g\in \hol(\cd)$, consider the weighted composition operator
$\cph^g f (z) = g(z) \cph f(z)$, $z\in \cd$.
Given a Banach space $X\subset \hol(\cd)$, the bounded operators $\cph^g: X \to \fk_\al^\infty(\cd)$
are characterizable in terms of $\al, d$ and the norms $\|\delta_z\|_{X^*}$, $z\in\cd$, where $\delta_z(f) = f(z)$.
In particular, for the unit disk $\Dbb$ of $\Cbb$ and $X\subset\hol(\Dbb)$,
similar schemes were used in \cite{Du11} and \cite{CT16} for $\cph^g: X \to Y$,
where $Y$ is a growth, weighted Bloch or Lipschitz space of arbitrary order.
The corresponding compact operators are described by related little-oh conditions.
So, we concentrate our attention on the case where $p=\infty$.

For $p=\infty$, as in the work of Ueki \cite{U07}, the characterizations of bounded (compact)
operators $\vtgp$ are based in \cite{Men13IEOT} on
Berezin type integral transforms; see Section~\ref{s_fk} for details.
We use approximation of radial weights by appropriate entire functions to
obtain more explicit descriptions of the bounded (compact) $\vtgp$ and similar operators
on $\fk_\al^\infty(\cd)$, $d\ge 1$, and on general growth spaces $\grw(\cd)$, $d\ge 1$.

\subsection{Growth spaces}
Let $\we: [0, +\infty)\to (0, +\infty)$ be {a weight function},
that is, let $\we$ be non-decreasing, continuous and unbounded.
Setting $\we(z) = \we(|z|)$ for $z\in\cd$, we extend $\we$
to {a radial weight} on $\cd$.
The corresponding growth space $\grw(\cd)$ consists of those $f\in\hol(\cd)$ for which
\[
\|f\|_{\grw} = \sup_{z\in\cd} \frac{|f(z)|}{\we(z)} < \infty.
\]
Given a set $X$ and functions $u, v: X\to (0,+\infty)$, we write $u\asymp v$
and we say that $u$ and $v$ are {equivalent} if
$C_1 u(x) \le v(x) \le C_2 u(x)$, $x\in X$,
for some constants $C_1, C_2 >0$. Replacing $\we$ in the definition of $\grw(\cd)$
by an equivalent weight function,
we clearly obtain the same space, with an equivalent norm.
In what follows, we also always assume that
$\lim_{t\to +\infty} t^{-k} \we(t) =\infty$ for all $k\in\Nbb$.
This condition excludes finite-dimensional spaces $\grw(\cd)$ of polynomial growth.

\subsubsection{Associated weight functions}
To work with arbitrary weight functions, we
use the notion of associated weight formally introduced in \cite{BBT98}.

\begin{df}
Let $v: [0, +\infty) \to (0, +\infty)$ be a weight function.
For $d\ge 1$, let $v_d(z) = v(|z|)$, $z\in\cd$, be the corresponding radial weight on $\cd$.
{The associated weight} $\vti_d$ is defined by
\[
\vti_d(z)
=\sup\{|f(z)|: f\in\hol(\cd),\ |f|\le v_d\ \mathrm{on\ } \cd\}, \quad z\in\cd.
\]
\end{df}
As observed in \cite{BBT98}, $\vti_1$ is a radial weight, hence,
the associated weight function $\vti_1: [0,+\infty) \to (0,+\infty)$ is uniquely defined.
Using compositions with unitary transformations of $\cd$, we conclude that
$\vti_d$ is radial for all $d\ge 1$.
Observe that a different notion of radial weight is used in \cite{BBT98} for $d\ge 2$.
Also, standard arguments guarantee that
the identity $\vti_d = \vti:= \vti_1$, $d\ge 2$,
holds for the associated weight functions.

\subsubsection{Essential weight functions}\label{ss_ess}
Let $\we: [0,+\infty) \to (0,+\infty)$ be an arbitrary weight function.
The weight function $\weti$ correctly defines
the corresponding associated weight on $\cd$ for any $d\ge 1$.
Also, the definition of the associated weight guarantees that $\grw(\cd)=\grass(\cd)$ isometrically.
However, working with $\grass(\cd)$, we obtain 
answers in terms of $\weti$.
So, as in \cite{BBT98}, we say that $\we$ is {essential} if $\we$ is equivalent to $\weti$ on $[0, +\infty)$.
Also, it is known and easy to see that $\widetilde{(\weti)} = \weti$, hence,
$\weti$ is always essential.
This observation allows to apply approximation results from \cite{AD16} to arbitrary weight functions.

\subsection{Results}
In this section, we formulate several typical results of the present paper.
Our first theorem is a general fact related to operators mapping into lattices,
spaces with minimal structure assumptions.
Let $\YY(\cd)$ be a linear space whose elements are functions $F:\cd\to\Cbb$.
We say that $\YY(\cd)$ is a lattice if the following property holds:
\textsl{Assume that $f, F:\cd\to\Cbb$ are continuous functions and $|f(z)|\le |F(z)|$, $z\in\cd$.
If $F\in\YY(\cd)$, then $f\in\YY(\cd)$}.
%

\begin{theorem}\label{t_wco}
Suppose that $\we: [0,+\infty) \to (0, +\infty)$ is a weight function,
$g\in\hol(\cd)$, $d\ge 1$, $\ph:\cd\to\cd$ is a holomorphic map,
and $\YY(\cd)$ is a lattice. Then the weighted composition operator
$\cph^g$ maps $\grw(\cd)$ into $\YY(\cd)$ if and only if
\begin{equation}\label{e_wco}
|g(z)| \weti (|\ph(z)|) \in\YY(\cd),
\end{equation}
where $\weti$ denotes the associated weight function.
\end{theorem}

In the present paper, $\YY(\cd)$ is usually a weighted $L^q$ space.
In particular, for Volterra type operators mapping into Fock type spaces, we have the following corollaries.

\begin{cory}\label{c_volt_grw_to_fk}
Let $\we$ be an arbitrary radial weight on $\cd$, $d\ge 1$.
Assume that $g\in\hol(\cd)$, $\ph:\cd\to\cd$ is a holomorphic map, $\alpha>0$ and $0<q\le\infty$.
Then the following properties are equivalent:
\begin{itemize}
  \item [(i)] $\vtgp: \grw(\cd) \to \fk_\al^q(\cd)$ is a bounded operator;
  \item [(ii)]
  \[
  \frac{|\rad g(z)|}{(1+|z|)^2} e^{-\frac{\al}{2} |z|^2} \weti(|\ph(z)|) \in L^q(\cd, \lbd).
  \]
\end{itemize}
\end{cory}


\begin{cory}\label{c_cmpg}
Let $\we$ be a radial weight on $\Cbb$.
Assume that $\ph, g\in\hol(\Cbb)$, $\alpha>0$ and $0<q\le\infty$.
Then the following properties are equivalent:
\begin{itemize}
  \item [(i)] $\cmpg: \grw(\Cbb) \to \fk_\al^q(\Cbb)$ is a bounded operator;
  \item [(ii)]
  \[
  \frac{|g^\prime(z) \ph^\prime(z)|}{1+|z|} e^{-\frac{\al}{2} |z|^2} \weti(|\ph(z)|) \in L^q(\Cbb, \lbarea).
  \]
\end{itemize}
\end{cory}

Corollary~\ref{c_cmpg} and analogous results are restricted to spaces of holomorphic functions on $\Cbb$,
since the corresponding proofs
depend on differentiation of compositions.

Also, we obtain similar explicit results for $\kom_g$, the companion to $\vot_g$ defined as
\[
\kom_g f(z) = \int_0^1 \rad f(tz) g(tz)\, \frac{dt}{t},\quad f\in\hol(\cd),\ z\in\cd.
\]
Berezin type characterizations of such operators
were obtained in \cite{Men16pota}.

\subsection{Organization of the paper}
Auxiliary facts are collected in Section~\ref{s_aux}.
Section~\ref{s_basics} contains basic results
related to the weighted composition operators from a growth space into a lattice.
In particular, we prove Theorem~\ref{t_wco} and we characterize $q$-Carleson measures for
$\fk_\al^\infty(\cd)$, $d\ge 1$,
and other growth spaces.
Corollaries~\ref{c_volt_grw_to_fk} and \ref{c_cmpg} are obtained in Section~\ref{s_fk}.
We prove related corollaries for
certain combinations of $\kom_g$ and $\cph$
defined on the growth Fock spaces.
 Also, we compare particular cases of the corollaries obtained and analogous results from \cite{Men13IEOT, Men16pota}.

\section{Auxiliary results}\label{s_aux}
\subsection{Approximation by entire maps}

\begin{theorem}[{\cite[Theorem~3.5]{AD16}}]\label{t_TropApp}
Let $\we: [0,+\infty) \to (0, +\infty)$ be a weight function.
Then $\we$ is essential if and only if there exists $n\in\Nbb$ and a holomorphic
map $f: \cd\to \Cbb^n$ such that $|f_1| + \dots +|f_n| \asymp \we$.
\end{theorem}

To work with explicit weight functions $\we$, we need a sufficient condition for being essential.
So, consider the logarithmic transform
\[
\Phi_\we(x) = \log \we(e^x), \quad -\infty < x < +\infty.
\]
If $\Phi_\we$ is a convex function, then $\we$ is called {log-convex}.
Observe that the property of being essential for $\we$ depends
only on the behavior of $\Phi_\we$ at $+\infty$.
In fact, Theorem~\ref{t_TropApp} implies that equivalence to a log-convex weight function
is necessary for being essential.
The following sufficient condition is somewhat more stringent.

\begin{theorem}[{\cite[paragraph after Example~1]{AD16}}]\label{t_ess_suff}
Let $\we: [0,+\infty) \to (0, +\infty)$ be a log-convex and $\mathcal{C}^2$-smooth weight function.
Assume that
\begin{equation}\label{e_ess_suff}
\liminf_{x\to +\infty} \Phi_\we^{\prime\prime}(x) >0.
\end{equation}
Then $\we$ is essential.
\end{theorem}

\subsection{Characterizations of the Fock spaces}
\begin{theorem}[see \cite{Hu13, U16}]\label{t_fk_eqnorms}
Let $\al>0$, $0<p\le \infty$ and let $f\in\hol(\cd)$, $d\ge 1$.
Then $f\in\fk_\al^p(\cd)$ if and only if
\[
\frac{|\rad f(z)|}{(1+|z|)^2} e^{-\frac{\al}{2}|z|^2} \in L^p(\cd, \lbd).
\]
\end{theorem}

\section{Basic results}\label{s_basics}
\subsection{Weighted composition operators}

\begin{proof}[Proof of Theorem~\ref{t_wco}]
First, assume that $f\in\grw(\cd)$ and \eqref{e_wco} holds, that is, $|g(z)| \weti (|\ph(z)|) \in\YY(\cd)$.
Since $\grw(\cd)$ isometrically equals $\grass(\cd)$, we have
\[
|\cph^g f(z)| \le \|f\|_{\grass(\cd)} |g(z)| \weti(|\ph(z)|) \in\YY(\cd).
\]
So, $\cph^g f\in\YY(\cd)$ by the definition of a lattice.

Secondly, assume that $\cph$ maps $\grw(\cd)=\mathcal{A}^{\weti}(\cd)$ into $\YY(\cd)$.
As indicated in Section~\ref{ss_ess}, the weight function $\weti$ is essential; thus, applying Theorem~\ref{t_TropApp},
we obtain functions $f_1, \dots, f_n \in \mathcal{A}^{\weti}(\cd)$ such that
\[
|f_1(z)|+ \cdots + |f_n(z)| \ge \weti(|z|), \quad z\in\cd.
\]
Hence, for $z\in\cd$, we have
\[
|g(z)| \weti(|\ph(z)|) \le \sum_{j=1}^n |g(z)||f_j(\ph(z))| = \sum_{j=1}^n |\cph^g f_j (z)| \in \YY(\cd).
\]
Therefore, \eqref{e_wco} holds by the definition of a lattice.
\end{proof}

\subsection{Carleson measures}\label{ss_carleson}
Let $0<q<\infty$ and let $X\subset\hol(\cd)$ be a linear space.
A positive Borel measure $\mu$ on $\cd$ is called $q$-Carleson for $X$ if $X\subset L^q(\cd, \mu)$.
Applying Theorem~\ref{t_wco} with $\YY(\cd)= L^q(\cd, \mu)$, $g\equiv 1$ and $\ph(z)\equiv z$,
we obtain the following result.

\begin{cory}\label{c_carleson}
Suppose that $\we: [0,+\infty) \to (0, +\infty)$ is a weight function,
$\mu$ is a positive Borel measure on $\cd$, and $0<q<\infty$.
Then $\mu$ is a $q$-Carleson measure for $\grw(\cd)$ if and only if
\begin{equation}\label{e_carleson}
\int_{\cd} \weti^q(|z|)\, d\mu(z) < \infty.
\end{equation}
\end{cory}

\subsubsection{Fock--Sobolev Carleson measures}
The Fock type spaces $\fk_{1, \beta}^p(\cd)$, $\beta\in\Rbb$, $0<p\le\infty$, are introduced in
\cite{CCK15}. In particular, $\fk_{1, \beta}^\infty(\cd) = \mathcal{A}^{\we_{1,\beta}}(\cd)$, where
$\we_{1,\beta}(t) = (1+t)^\beta e^{\frac{1}{2}t^2}$, $0\le t< +\infty$.
As shown in \cite{CCK15}, any Fock--Sobolev space $\fk_{1, \beta, s}^p(\cd)$ of fractional order $s\in\Rbb$
coincides with $\fk_{1, \beta-s}^p(\cd)$, $0<p\le \infty$, $\beta\in\Rbb$.
This fact allows to extend most of the results of the present paper to $\fk_{1, \beta}^\infty(\cd)$, $\beta\in\Rbb$.
We restrict our attention to characterizations of $q$-Carleson measures for $\fk_{1, \beta}^p(\cd)$.
For $0<p<\infty$, the problem is solved in \cite{CCK15}. So, we consider the case where $p=\infty$.

\begin{cory}\label{c_carl_f_sobolev}
  Suppose that $d\ge 1$, $\beta\in\Rbb$, $0<q<\infty$ and $\mu$ is a positive Borel measure on $\cd$.
  Then $\mu$ is a $q$-Carleson measure for $\fk_{1, \beta}^\infty(\cd)$ if and only if
\[
\int_{\cd} (1+|z|^\beta)^q e^{\frac{q}{2}|z|^2}\, d\mu(z) < \infty.
\]
\end{cory}
\begin{proof}
  The logarithmic transform $\Phi_{\we_{1, \beta}}(x) = \beta\log(1+ e^x) + \frac{1}{2}e^{2x}$, $-\infty< x < +\infty$,
  clearly satisfies condition \eqref{e_ess_suff}. So, it suffices to apply Theorem~\ref{t_ess_suff}
  and Corollary~\ref{c_carleson}.
\end{proof}

\subsubsection{Carleson measures for $\fk_\al^\infty(\cd)$}
We have $\fk_\al^\infty(\cd) = \mathcal{A}^{\we_\al}(\cd)$
with $\we_\al(t)  = e^{\frac{\al}{2}t^2}$, $\al>0$, $0\le t < \infty$.
So, Theorem~\ref{t_ess_suff} and Corollary~\ref{c_carleson}
guarantee that $\mu$ is a $q$-Carleson measure for $\fk_\al^\infty(\cd)$, $d\ge 1$,
if and only if
\[
\int_{\cd} e^{\frac{\al q}{2}|z|^2}\, d\mu(z) < \infty.
\]
For $d=1$, the above condition was obtained in \cite{Men13IEOT} by a different method.

\subsection{Volterra type operators}
Recall that the Volterra type operator $\vtgp$
is defined as
\[
(\vtgp f)(z)= \int_0^1 f(\ph(tz)) \frac{\rad g(tz)}{t}\, dt, \quad f\in\hol(\cd),\ z\in\cd,
\]
where $g\in\hol(\cd)$ and $\ph:\cd\to \cd$ is a holomorphic map.
Direct calculations guarantee that
\begin{equation}\label{e_hu2003}
\rad\vtgp f (z) = f(\ph(z)) \rad g(z),\quad z\in\cd,
\end{equation}
for all $f,g\in \hol(\cd)$; see \cite{Hu03}.

\begin{cory}\label{c_rad_volt}
Suppose that $\we: [0,+\infty) \to (0, +\infty)$ is a weight function,
$g\in\hol(\cd)$, $\ph:\cd\to\cd$ is a holomorphic map,
and $\YY(\cd)$ is a lattice. Then $\rad\vtgp$
maps $\grw(\cd)$ into $\YY(\cd)$ if and only if
$|\rad g(z)| \weti (|\ph(z)|) \in\YY(\cd)$.
\end{cory}
\begin{proof}
It suffices to apply \eqref{e_hu2003} and Theorem~\ref{t_wco}.
\end{proof}

\subsection{Compact operators}
Characterizations of the compact operators considered in Theorem~\ref{t_wco}
and its corollaries depend on the structure of the corresponding lattice $\YY(\cd)$.
If $\YY(\cd)$ is an $L^\infty$ space, then, as indicated in Section~\ref{ss_fock},
the little-oh versions of the boundedness conditions describe the corresponding compact operators.
If $\YY(\cd)= L^q(\cd, \mu)$ with $0<q<\infty$, then the bounded operators under consideration
are automatically compact.
For example, given $0<q<\infty$ and a positive Borel measure $\mu$ on $\cd$,
standard arguments show that the identity operator $\mathrm{Id}: \grw(\cd)\to L^q(\cd, \mu)$
is compact if and only if \eqref{e_carleson} holds.
So, in what follows, we restrict our attention to the boundedness conditions.

\section{Operators on the growth Fock spaces}\label{s_fk}
\subsection{Corollary~\ref{c_volt_grw_to_fk} and related results}
\begin{proof}[Proof of Corollary~\ref{c_volt_grw_to_fk}]
By Theorem~\ref{t_fk_eqnorms}, property~(i) from Corollary~\ref{c_volt_grw_to_fk} holds
if and only if
\[
\frac{(\rad\vtgp f)(z)}{(1+|z|)^2} e^{-\frac{\al}{2} |z|^2}
\in L^q(\cd, \lbd)
\]
for all $f\in\grw(\cd)=\mathcal{A}^{\weti}(\cd)$.
Let $\YY(\cd)$ consist of those measurable $F:\cd\to\Cbb$ for which
\[
\frac{|F(z)|}{(1+|z|)^2} e^{-\frac{\al}{2} |z|^2}
\in L^q(\cd, \lbd).
\]
Corollary~\ref{c_rad_volt} guarantees that (i) is equivalent
to the following property:
\begin{equation}\label{e_volt_grw_to_fk}
\frac{|\rad g(z)| \weti(|\ph(z)|)}{(1+|z|)^2} e^{-\frac{\al}{2} |z|^2} \in L^q(\cd, \lbd),
\end{equation}
as required.
\end{proof}

\subsubsection{Berezin type characterizations}
The normalized kernel function $k_{w, \al}$ for the Hilbert space $\fk_\al^2(\cd)$ is given by the identity
\begin{equation}\label{e_berezin_ker}
k_{w, \al}(z) = e^{\al\langle z, w \rangle - \al |w|^2 /2}, \quad z,w \in\cd.
\end{equation}
To study $\cph^g$, $\vot_g^\ph$ and related operators on Fock type spaces,
Ueki \cite{U07} and Mengestie \cite{Men13IEOT, Men14JGA, Men16pota} introduce
Berezin type integral transforms with the help of
$k_{w, \al}(z)$.
For example, the bounded (compact) Volterra type operators
$\vtgp: \fk_\al^\infty(\Cbb) \to \fk_\al^q(\Cbb)$
are characterized by the following theorem.

\begin{theorem}[{\cite[Theorem~2.3]{Men13IEOT}}]\label{t_men_ieot23}
Assume that $g, \ph\in\hol(\Cbb)$, $\alpha>0$ and $0<q<\infty$.
Then the following properties are equivalent:
\begin{itemize}
  \item [(i)] $\vtgp: \fk_\al^\infty(\Cbb) \to \fk_\al^q(\Cbb)$ is a bounded operator;
  \item [(ii)] $\vtgp: \fk_\al^\infty(\Cbb) \to \fk_\al^q(\Cbb)$ is a compact operator;
  \item [(iii)]
  \[
  \int_{\Cbb}\int_{\Cbb} \frac{|k_{w, \al}(\ph(z)) g^\prime(z)|^q}{(1+|z|)^q} e^{-\frac{\al q}{2} |z|^2}
  \, d\lbarea(z) \, d\lbarea(w) < \infty.
  \]
\end{itemize}
\end{theorem}

\subsubsection{Comparison of Corollary~\ref{c_volt_grw_to_fk} and Theorem~\ref{t_men_ieot23}}
We have $\fk_\al^\infty(\cd) = \mathcal{A}^{\we_\al}(\cd)$ with $\we_\al(t)= e^{\frac{\al}{2} t^2}$,
$\al>0$, $0\le t<+\infty$.
As indicated in Section~\ref{ss_carleson}, Theorem~\ref{t_ess_suff} guarantees that
$\weti_\al\asymp \we_\al$.
Hence, \eqref{e_volt_grw_to_fk} with $\we=\we_\al$ reads as
\[
\frac{|\rad g(z)|}{(1+|z|)^2} e^{\frac{\al}{2}(|\ph(z)|^2 - |z|^2)} \in L^q(\cd, \lbd).
\]
If $d=1$, then $\rad g(z) = zg^\prime(z)$, $z\in\Cbb$.
Thus, the following condition is equivalent to properties (i--iii) from Theorem~\ref{t_men_ieot23}:
\begin{equation}\label{e_volt_dim1}
\frac{|g^\prime (z)|}{1+|z|} e^{\frac{\al}{2}(|\ph(z)|^2 - |z|^2)} \in L^q(\Cbb, \lbarea).
\end{equation}
If $g(z)$ is a constant, then $\vtgp={0}$.
If $g(z)$ is not a constant, then \eqref{e_volt_dim1} allows to make certain explicit
conclusions about those symbols $g$ and $\ph$ for which
$\vtgp: \fk_\al^\infty(\Cbb)\to \fk_\al^q(\Cbb)$, $0<q<\infty$,
is bounded or, equivalently, compact.

\begin{cory}
Let $g,\ph\in\hol(\Cbb)$, $\al>0$ and $0<q<\infty$.
\begin{itemize}
  \item [(i)]
Assume that $g(z)$ grows as a power function of degree at least $2$ and
$\vtgp: \fk_\al^\infty(\Cbb)\to \fk_\al^q(\Cbb)$ is bounded.
Then $\ph(z) = \beta z+ \gamma$ with $|\beta|<1$.
  \item [(ii)]
Assume that $g(z)=az+b$, $a\neq 0$, and $0<q\le 2$. Then
$\vtgp: \fk_\al^\infty(\Cbb)\to \fk_\al^q(\Cbb)$ is bounded (compact) if and only if
$\ph(z) = \beta z+ \gamma$ with $|\beta|<1$.
  \item [(iii)]
Assume that $g(z)=az+b$, $a\neq 0$, and $2<q<\infty$. Then
$\vtgp: \fk_\al^\infty(\Cbb)\to \fk_\al^q(\Cbb)$ is bounded (compact) if and only if
$\ph(z) = \beta z+ \gamma$ with $|\beta|\le 1$.
\end{itemize}
\end{cory}

\subsection{Corollary~\ref{c_cmpg} and related results}
\begin{proof}[Proof of Corollary~\ref{c_cmpg}]
Observe that $\cmpg f (z) = \int_0^{\ph(z)} f(\xi) g^\prime(\xi)\, d\xi$, $z\in\Cbb$.
So, Theorem~\ref{t_fk_eqnorms} guarantees that $\cmpg$ maps $\grw(\Cbb)$ into $\fk_\al^q(\Cbb)$
if and only if
\[
\frac{|f(\ph(z)) g^\prime(\ph(z)) \ph^\prime(z)|}{1+|z|} e^{-\frac{\al}{2} |z|^2}
\in L^q(\Cbb, \lbarea)\quad \textrm{for all\ } f\in\mathcal{A}^{\weti}(\Cbb).
\]
Since $\weti$ is essential, the above condition is equivalent
to the following one:
\begin{equation}\label{e_cmpg}
\frac{|g^\prime(\ph(z)) \ph^\prime(z)|}{1+|z|} \weti(|\ph(z)|) e^{-\frac{\al}{2} |z|^2} \in L^q(\Cbb, \lbarea)
\end{equation}
by Theorem~\ref{t_wco}. The proof of Corollary~\ref{c_cmpg} is finished.
\end{proof}

\subsubsection{Berezin type characterizations}
By \cite[Theorem~2.4]{Men13IEOT}, the operator
$\cmpg: \fk_\al^\infty(\Cbb)\to\fk_\al^q(\Cbb)$, $0<q<\infty$, is bounded (compact) if and only if
\begin{equation}\label{e_cmpg_men}
\int_{\Cbb} \int_{\Cbb}
\frac{|k_{w,\al}(\ph(z)) g^\prime(\ph(z)) \ph^\prime(z)|^q}{\left((1+|z|) e^{\frac{\al}{2}|z|^2} \right)^q}
\,d\lbarea(z) \,d\lbarea(w) < \infty,
\end{equation}
where $k_{w, \al}(z)$ is defined by \eqref{e_berezin_ker}.

\subsubsection{Comparison of conditions \eqref{e_cmpg} and \eqref{e_cmpg_men}}
For $\we_\al(t)  = e^{\frac{\al}{2}t^2}$, $\al>0$, $0\le t < \infty$,
condition \eqref{e_cmpg} rewrites as
\begin{equation}\label{e_cmpg_fk}
\frac{|g^\prime(\ph(z)) \ph^\prime(z)|}{1+|z|} e^{\frac{\al}{2}(|\ph(z)|^2 - |z|^2)} \in L^q(\Cbb, \lbarea).
\end{equation}
In particular, \eqref{e_cmpg_men} and \eqref{e_cmpg_fk} are equivalent.
If $g$ is a constant, then $\kom_g^\ph={0}$.
Using \eqref{e_cmpg_fk}, we also make the following explicit conclusions.

\begin{cory}
Let $g,\ph\in\hol(\Cbb)$, $\al>0$ and $0<q<\infty$.
\begin{itemize}
  \item [(i)]
Assume that $g(z)$ is not a constant and
$\kom_g^\ph: \fk_\al^\infty(\Cbb)\to \fk_\al^q(\Cbb)$ is bounded.
Then $\ph(z) = \beta z+ \gamma$ with $|\beta|<1$.
  \item [(ii)]
Assume that $g(z)=az+b$, $a\neq 0$. Then
$\kom_g^\ph: \fk_\al^\infty(\Cbb)\to \fk_\al^q(\Cbb)$ is bounded (compact) if and only if
$\ph(z) = \beta z+ \gamma$ with $|\beta|<1$.
\end{itemize}
\end{cory}

\subsection{Volterra companion integral operators}
Let $g\in\hol(\Cbb)$. Then the operator $\kom_g$, companion to $\vot_g$, is defined by
\[
\kom_g f(z) = \int_0^z f^\prime(\xi) g(\xi)\, d\xi, \quad \xi\in\Cbb.
\]
Companions of $\vtgp$ and $\cmpg$ are defined as
\[
  \kom_g^\ph f(z) =
\int_0^z f^\prime (\ph(\xi)) g(\xi)\, d\xi,\quad
  \kkom_{\ph, g} f(z) =
\int_0^{\ph(z)} f^\prime(\xi) g(\xi)\, d\xi, \quad z\in\Cbb.
\]
The bounded and compact operators $\kom_g^\ph, \kkom_{\ph, g}: \fk_\al^p(\Cbb)\to \fk_\al^q(\Cbb)$
are characterized in \cite{Men16pota} for all $0<p,q\le +\infty$ in terms of Berezin type transforms.
We apply Theorem~\ref{t_TropApp} to give more explicit descriptions for $p=\infty$.

\begin{cory}\label{c_kom_gp}
Assume that $g, \ph\in\hol(\Cbb)$, $\al>0$ and $0<q\le \infty$.
Then the following properties are equivalent:
\begin{itemize}
  \item [(i)] $\kom_g^\ph: \fk_\al^\infty(\Cbb)\to \fk_\al^q(\Cbb)$ is a bounded operator;
  \item [(ii)]
  \[
\frac{1+|\ph(z)|}{1+|z|} |g(z)| e^{\frac{\al}{2} (|\ph(z)|^2 - |z|^2)} \in L^q(\Cbb, \lbarea).
  \]
\end{itemize}
\end{cory}
\begin{proof}
First, Theorem~\ref{t_fk_eqnorms} guarantees that property~(i) holds if and only if
\begin{equation}\label{e_kom_eqnorm}
\frac{|f^\prime(\ph(z)) g(z)|}{1+|z|} e^{-\frac{\al}{2} |z|^2} \in L^q(\Cbb, \lbarea)\quad\textrm{for all\ } f\in\fk_\al^\infty(\Cbb).
\end{equation}
If
$f\in\fk_\al^\infty(\Cbb)$, then $|f^\prime(\xi)| \le (1+|\xi|)e^{\frac{\al}{2} |\xi|^2}$, $\xi\in\Cbb$,
by Theorem~\ref{t_fk_eqnorms} with $p=\infty$. Therefore, (ii) implies (i).

Secondly, assume that (i) holds. Theorem~\ref{t_ess_suff} is applicable to the weight function
$\we_{\al,1}(t) = (1+t) e^{\frac{\al}{2}t^2}$, $0\le t<\infty$.
Hence, Theorem~\ref{t_TropApp} provides functions $h_1, \dots, h_n \in \hol(\Cbb)$
such that
\begin{equation}\label{e_deriv_we}
|h_1(z)|+ \dots + |h_n(z)| \asymp \we_{\al, 1}(|z|), \quad z\in\Cbb.
\end{equation}
Put
\[
f_j(z)=\int_0^z h_j(\xi)\, d\xi,\quad z\in\Cbb.
\]
Then $f_j^\prime = h_j$ and $f_j\in \fk_\al^\infty(\Cbb)$, $j=1,\dots, n$, by \eqref{e_deriv_we}
and Theorem~\ref{t_fk_eqnorms} with $p=\infty$.
So, putting $f=f_j$, $j=1, \dots, n$, in \eqref{e_kom_eqnorm},
taking the sum and using \eqref{e_deriv_we},
we conclude that (i) implies (ii).
\end{proof}


\begin{cory}
Let $\ph, g\in\hol(\Cbb)$, $\al>0$ and $0<q<\infty$.
Assume that $g(z)\not\equiv 0$ and
$\kkom_{\ph, g}: \fk_\al^\infty(\Cbb)\to \fk_\al^q(\Cbb)$ is bounded.
Then $\ph(z) = \beta z+ \gamma$ with $|\beta|<1$.
\end{cory}

The proof of the following explicit description for $\kkom_{\ph, g}$ is similar to that of Corollary~\ref{c_kom_gp}.

\begin{cory}\label{c_kkom_pg}
Assume that $\ph, g \in\hol(\Cbb)$, $\al>0$ and $0<q\le \infty$.
Then $\kkom_{\ph, g}: \fk_\al^\infty(\Cbb)\to \fk_\al^q(\Cbb)$ is a bounded operator if and only if
  \[
\frac{1+|\ph(z)|}{1+|z|} |g(\ph(z)) \ph^\prime(z)| e^{\frac{\al}{2} (|\ph(z)|^2 - |z|^2)} \in L^q(\Cbb, \lbarea).
  \]
\end{cory}


\end{document}